\documentclass[11pt]{amsart}
\usepackage[T1]{fontenc}
\usepackage[text={6in,8.5in},centering]{geometry}
\usepackage[colorlinks=true,citecolor=black!60!green,linkcolor=black!60!red,filecolor=black!60!cyan,urlcolor=black!60!magenta]{hyperref}
\usepackage{xcolor}
\usepackage{amssymb,amsmath,amsthm,mathtools,extpfeil}
\usepackage[all,cmtip]{xy}
\usepackage{tikz,tikz-3dplot}
\usepackage[shortlabels]{enumitem}

\usepackage{caption}
\usepackage{subcaption}

\newtheorem{thm}[equation]{Theorem}
\newtheorem*{thm*}{Theorem}
\newtheorem*{thmA}{Main Theorem}
\newtheorem{lem}[equation]{Lemma}
\newtheorem{prop}[equation]{Proposition}
\newtheorem{cor}[equation]{Corollary}

\theoremstyle{definition}

\newtheorem*{rmk}{Remark}
\newtheorem*{rmks}{Remarks}

\numberwithin{equation}{section}

\DeclareMathOperator{\st}{st}
\DeclareMathOperator{\lk}{lk}
\DeclareMathOperator{\bs}{bs}

\DeclareMathOperator{\Int}{int}

\DeclareMathOperator{\Th}{Th}

\DeclareMathOperator{\Lip}{Lip}

\newcommand{\ph}{\varphi}
\newcommand{\epsi}{\varepsilon}
\newcommand{\LL}{P}

\newcommand{\ZZ}{\mathbb{Z}}

\makeatletter
\let\@wraptoccontribs\wraptoccontribs
\makeatother

\begin{document}
\title{Quantitative PL bordism}
\author[F.~Manin]{Fedor Manin}
\author[Sh.~Weinberger]{Shmuel Weinberger}
\address[F.~Manin]{Department of Mathematics, University of Toronto, ON, Canada}
\email{manin@math.toronto.edu}
\address[Sh.~Weinberger]{Department of Mathematics, University of Chicago, IL, United States}
\email{shmuel@math.uchicago.edu}
\begin{abstract}
  We study PL bordism theories from a quantitative perspective.  Such theories include those of PL manifolds, ordinary homology theory, as well as various more exotic theories such as bordism of Witt spaces.  In all these cases we show that a null-bordant cycle of bounded geometry with $V$ simplices has a filling of bounded geometry whose number of simplices is slightly superlinear in $V$.  This bound is similar to that found in our previous work on smooth cobordism.
\end{abstract}
\maketitle

\section{Introduction}

Suppose a compact manifold $M$ is \emph{nullcobordant}, i.e.\ is the boundary of some compact manifold $W$.  How complicated must $W$ be?  This is not just one question, but a whole family of questions, depending on the type of manifold that we allow $M$ and $W$ to be, and how we measure the complexity of $W$.  

To our knowledge, Gromov is the first to ask such questions; he discusses several variations in \cite[\S$5\frac{5}{7}$]{GrPC} and \cite{GrQHT}.  In the case that $M$ and $W$ are smooth, he gave as a possible measure of complexity the infimal volume of a metric given a fixed bound on the local geometry (say, sectional curvatures $|K| \leq 1$ and injectivity radius at least $1$).  He suggested that the minimal complexity of a $W$ which fills $M$ should be linear in the complexity of $M$; a slightly superlinear bound ($O(V^{1+\epsi})$ for any $\epsi>0$) was proved in \cite{CDMW}.

In this paper we prove an analogous bound in the PL category.  In this setting, ``bounded geometry of type $L$'' means that every vertex is incident to at most $L$ simplices, and volume is measured by the total number of top-dimensional simplices.
\begin{thm} \label{prelim-main}
  Let $M$ be a PL $k$-manifold with bounded geometry of type $L$ and $V$ top-dimensional simplices.  Then if $M$ is a PL boundary, it bounds a manifold $W$ with bounded geometry of type $L'$ (depending on $k$ and $L$) so that the number of simplices in $W$ is bounded by a function (again depending inexplicitly on $k$ and $L$) which is $O(V^{1+\epsi})$ for every $\epsi>0$.
\end{thm}
Here, and in further similar statements, the implicit constant depends (inexplicitly) on $\epsi$; for example, the theorem does not exclude the possibility of a linear-times-polylog bound.

\subsection{Discussion}
The proof of the bound in \cite{CDMW} built on the original work of Ren\'e Thom \cite{Thom0,Thom} relating cobordism to the homotopy groups of the Thom space of the universal bundle over a Grassmannian.  This Thom space is a specific finite complex $X$, and to prove the quantitative bound one must relate the Lipschitz constant of a nullhomotopic map $f:S^N \to X$ to the Lipschitz constant of its most efficient nullhomotopy.  Even talking about Lipschitz constants requires a metric on $X$, but for an asymptotic computation the specific metric doesn't matter because all the reasonable (Riemannian or piecewise Riemannian) metrics are Lipschitz homotopy equivalent.

The method of proof is quite general, and although this is not pointed out in that paper, it applies to oriented and non-oriented cobordism, as well as spin and complex cobordism and any similar variant.  The reason is that in all such cases Thom's method gives the appropriate reduction to the homotopy theory of finite polyhedra, and the spaces that arise are particularly simple from the point of view of rational homotopy theory.

However, this method does not apply to PL cobordism, because the classifying spaces are not close relatives of compact Lie groups.  While they are abstractly homotopy equivalent to spaces with finite skeleta, we do not know of a natural way of representing them as such.   This paper studies the complexity of PL nullcobordisms, with an eye towards developing techniques that can apply in situations where the traditional methods of geometric topology (such as classifying spaces, h-principles, etc.) reduce geometric problems to the homotopy theory of infinite complexes.  

For infinite complexes, one can choose metrics that are inequivalent at large scales, and different choices can lead to very different measures of complexity.  For example, a model for $K(\mathbb Z,n)$ of finite type (i.e.\ with finitely many simplices in each dimension) and a simplicial group model (with infinitely many) lead to very different ``quantitative cohomology theories''.  A related situation arises in K-theory, whose concrete manifestation is the following.  Suppose one is interested in vector bundles of dimension $n$ over a finite complex $X$ of lower dimension.  The number of such bundles which admit a connection with bounded curvature $K$ grows as a polynomial in $K$ whose degree is related to the cohomology of $X$.  However, the naturally isomorphic set of stable vector bundles over $X$ can be realized by bundles (of varying dimension) with a uniform bound on curvature; if one takes curvature as the measure of complexity, then infinitely many classes are realized with bounded complexity.

Now, let us turn more specifically to the topic of this paper.  A reasonable notion of the complexity of a PL manifold $M$ is the number of simplices in a triangulation.  We could consider the problem: given a PL $k$-manifold $M$ with $V$ simplices, how many simplices are contained in the simplest $W$ whose boundary is $M$?

It is true, but not completely obvious, that the complexity of a nullcobordism can be bounded by a recursive (i.e.\ computable) function of the complexity of $M$.  This is equivalent to the statement that the PL bordism problem is algorithmically decidable, and that follows from the deep work that is explained e.g.\ in \cite[Ch. 14]{MadMil}.  One basic step in the analysis is the structure of $G/PL$\footnote{The homotopy fiber of $BPL \to BG$, the map from the classifying space for PL structures to that for spherical fibrations.}, which is computed using surgery theory and the Poincar\'e conjecture.  The algorithmic nature of these ingredients seems quite difficult to unravel.

One would have been led directly to the Poincar\'e conjecture in thinking about this question when considering whether a candidate $M$ that someone hands you is actually a PL manifold.  Maybe your interlocutor is playing a joke on you?  You can be sure that $M$ is a homology manifold, but it is much harder to know that $M$ is a manifold---you would need to see that the links of simplices are spheres.  Checking that a manifold is actually a sphere, if you think it should be, is exactly what the Poincar\'e conjecture is about. 

Despite first impressions, it is not quite necessary to complete an analysis of the algorithmic complexity of the Poincar\'e conjecture to handle the bordism question: one can skirt around it by producing a polyhedron bounded by $M$ which, if $M$ is a manifold, happens to be a manifold with boundary.  A strategy of this kind can work in some settings.

Regardless, we do not know how to deal with the PL bordism problem for arbitrary triangulations (although we do have a program that, if it succeeds, would give a tower of exponentials bound for the rise in complexity).  At the same time, there is no indication that the true answer is any more nonlinear than it is for smooth bordism.

Instead we consider the case of PL manifolds of bounded geometry.  That is, we consider PL manifolds with some bound on the number of simplices incident to each vertex, and build cobordisms constrained by a similar, perhaps larger bound.  This condition is similar to the smooth locally bounded geometry condition in that the amount of ``local information'' is bounded.  In fact, there is a closer correspondence: any smooth manifold of bounded geometry has a bounded geometry triangulation at scale $\sim 1$, and any smoothable PL manifold of bounded geometry admits a smoothing with locally bounded geometry and with simplices of bounded
volume; see Appendix \ref{smoothing}.

Moreover, in every dimension $d$, there is an $L(d)$ such that every PL $d$-manifold is PL homeomorphic to one with bounded geometry of type $L(d)$.  (However, given a manifold with $V$ simplices, the number of subdivisions required to implement a PL isomorphism with such a manifold grows faster than any computable function of $V$.)  So bordism of PL manifolds with bounded geometry is equivalent, as far as pure topology is concerned, to bordism of PL manifolds.

As in the smooth case, our methods work in a number of settings: they extend from PL manifolds to pseudomanifolds whose singularities come from a prescribed family (see \S\ref{S:singularities} for a precise definition), as well as to bordism homology of spaces other than a point.  (Recall that manifolds equipped with a map to a space $Y$, modulo cobordisms with a map to $Y$, form a generalized homology theory for $Y$ in the sense of Eilenberg--Steenrod.  This is also true for the more general setting we consider, see \cite{BRS}.)  The full statement of our main theorem is therefore more general than Theorem \ref{prelim-main}:



\begin{thmA}
  Let $\mathcal M$ be a category of PL manifolds with prescribed singularities.

  For every $k$, $L$, and $\epsi>0$, there are constants
  $C=C(\mathcal M,k,L,\epsi)$ and $L'=L'(\mathcal M,k,L)$ such that every PL
  nullcobordant triangulated $k$-dimensional $\mathcal M$-manifold $X$ with $V$
  $k$-simplices and bounded geometry of type $L$ has a filling with at most
  $CV^{1+\epsi}$ simplices and bounded geometry of type $L'$.  This holds for
  both oriented and unoriented cobordism.

  More generally, for a finite simplicial complex $Y$, there is a constant
  $C=C(\mathcal M,k,L,\epsi,Y)$ such that if $X$ is as above and $f:X \to Y$ is
  a simplicial map representing $0$ in the bordism homology group
  $\Omega^{\mathcal M}_k(Y)$, then there is a simplicial extension
  $\tilde f:W \to Y$ to a filling of $X$ with at most $CV^{1+\epsi}$ simplices
  and bounded geometry of type $L'$.
\end{thmA}
Note that the constants depend in an inexplicit way on the bound on geometry.
This is a point of contrast with the smooth category, where we can always change
the bound on geometry by rescaling, which automatically scales the nullcobordism
as well.


\subsection{Overview of proof}
As in \cite{CDMW}, we follow the usual proof of the classification of manifolds
up to cobordism, bounding the complexity of each step.  The main additional
technical challenge stems from the fact that while classifying spaces for PL
cobordism theories exist by Brown representability, unlike in the smooth case,
there are no nice explicit models of finite type that we can use to construct
our classifying maps.  We resolve this by building the models we need, but these
models may not actually be homotopy equivalent to the relevant classifying
spaces; they do, however, have enough topological resemblance for our purposes.

Let $V$ be the number of $k$-simplices of the manifold $X$.  We proceed as
follows:
\begin{enumerate}
\item Embed $X$ in $S^n$, with some control over the shape of a regular
  neighborhood.  Here we can choose any sufficiently large $n$.
\item This induces an $\ph(V)$-Lipschitz map from $S^n$ to a fixed compact
  subspace of a kind of Thom space for $\mathcal M$-manifolds.
\item One constructs a $\psi(\ph(V))$-Lipschitz extension of this map to
  $D^{n+1}$, with image in a larger fixed compact subspace of this Thom space.
\item From a simplicial approximation of this nullhomotopy, one can extract an
  $\mathcal M$-manifold embedded in $D^{n+1}$ which fills $X$ and whose number of
  simplices is bounded by the number of simplices in the approximation and is
  therefore $O(\psi(\ph(V))^{n+1})$.
\item After the previous step, the bound on the local geometry of the resulting
  filling $W$ depends on $n$.  We tidy up by applying local modifications to
  $W$, making the bound on the geometry independent of $n$ at the expense of
  increasing the number of simplices by an additional multiplicative constant.
\end{enumerate}
The final estimate depends on bounds on the functions $\ph$ and $\psi$.  Our
results from the appendix to \cite{CDMW} suffice to show that
\[\ph(V) \leq C(k,n)V^{\frac{1}{n-k}}(\log V)^{2k+2},\]
and we will show that $\psi(L)$ is at most linear.  This gives an overall bound
of $O(V^{1+O(1/n)})$ on the size of the resulting filling; since $n$ can be arbitrarily large, this gives our desired bound.

\subsection{Notation}
We collect here some notation and terminology for combinatorial operations on simplicial complexes which is used throughout the rest of the paper.

The \emph{join} of two abstract simplicial complexes $X$ and $Y$, written $X * Y$, is the complex whose simplices are $\Delta \sqcup \Delta'$, where $\Delta$ is a simplex of $X$ and $\Delta'$ is a simplex of $Y$.  Geometrically, the join may be thought of as the union of all line segments joining points of $X$ to points of $Y$.

Given a simplicial complex $X$, the \emph{cone} of $X$, written $cX$, is its join with a point.  The \emph{suspension} of $X$, written $SX$, is its join with $S^0$ (i.e.~two points).  We may also write $S^iX$ for the $i$ times suspension, i.e.~$S^iX=\underbrace{S \cdots S}_{i\text{ times}}X$.

Given a simplicial complex $X$ and a simplex $\Delta$ within it, the \emph{star} of $\Delta$, denoted $\st\Delta$, is the union of all simplices of $X$ which contain $\Delta$.  The \emph{open star} is the interior of the star.  The \emph{link} of $\Delta$, denoted $\lk\Delta$, is the subcomplex of $X$ such that $\Delta * \lk\Delta = \st\Delta$.

Finally, given an abstract simplicial complex $X$, the \emph{barycentric subdivision} $\bs(X)$ is the simplicial complex whose vertices are in bijection with simplices of $X$ and whose simplices are in bijection with ascending chains of simplices of $X$.  Geometrically, this corresponds to dividing each simplex
\[\Delta^k=\{t_0v_0+\cdots+t_kv_k : t_i \in [0,1], \sum_i t_i=1\}\]
into $(k+1)!$ subsets corresponding to permutations $\sigma$ of $\{0,1,\ldots,k+1\}$: the simplex
\[\Delta^k_\sigma=\{t_0v_0+\cdots+t_kv_k : 1 \geq t_{\sigma(0)} \geq t_{\sigma(1)} \geq \cdots \geq t_{\sigma(k)} \geq 0, \sum_i t_i=1\}\]
is indexed by the chain
\[\{\sigma(0)\} \subset \{\sigma(0),\sigma(1)\} \subset \cdots \subset \{\sigma(0),\ldots,\sigma(k)\}=\Delta^k.\]

\subsection*{Acknowledgements}
We are thankful to two anonymous referees for their helpful comments.

\section{Examples and corollaries}

We now discuss the implications of our results for particular categories of
manifolds with singularities.

\subsection{PL manifolds}
The most obvious such category is usual PL manifolds, with no singularities.
Here we remark upon a prior result of F.~Costantino and D.~Thurston:
\begin{thm*}[{\cite[Theorem 5.2]{CoTh}}]
  Every $3$-manifold $M$ with a triangulation with $V$ $3$-simplices has a
  filling $W$ which has bounded geometry and $O(V^2)$ simplices.
\end{thm*}
Note that in this theorem, $M$ is not required to have bounded geometry.
However, in dimension 3 this doesn't matter, and our main theorem implies a
stronger result:
\begin{cor} \label{cor:M3}
  Every $3$-manifold $M$ with a triangulation with $V$ $3$-simplices has a
  filling $W$ which has bounded geometry and $f(V)$ $4$-simplices, where
  $f(V)=O(V^{1+\epsi})$ for every $\epsi>0$.
\end{cor}
\begin{proof}
  It suffices to show that every $3$-manifold $M$ with $V$ $3$-simplices has a
  triangulation of bounded geometry with $O(V)$ simplices.  Such a triangulation
  can be produced as follows.  This strategy was outlined by Gromov
  \cite[\S$5\frac{5}{7}$II$'$]{GrPC}.

  First, take the unbounded triangulation and barycentrically subdivide once.

  Now let $\mathcal V_1$ be the set of vertices which come from the barycenters
  of edges.  The star of each such vertex looks like $ScZ$ (the suspension of
  the cone on $Z$), where $Z$ is a circle with some number $n$ of edges.  It is
  easy to see that $Z$ can be filled with a disk $D$ with $O(n)$ triangles, at
  most $7$ of which meet at a vertex, and at most $3$ of which meet at a vertex
  on $Z$.  So for each vertex in $\mathcal V_1$, we replace the subcomplex $ScZ$
  with $SD$, getting a PL homeomorphic complex $M'$.

  Now let $\mathcal V_2$ be the set of vertices of $M'$ that correspond to the
  original vertices of $M$.  The link of such a vertex in $M'$ is a
  triangulation of $S^2$ with $n$ faces and at most $9$ faces meeting at every
  vertex.  Such a triangulation again has a filling with bounded geometry and
  $O(n)$ $3$-simplices.  Gromov sketches the proof in
  \cite[\S$5\frac{5}{7}$II$''$]{GrPC}: after smoothing out the triangulation,
  one gets a Riemannian metric on $S^2$ with curvature bounded below by some
  $-K$.  By a theorem of Alexandrov \cite[\S XII.2]{Alex}, such an $S^2$ embeds
  isometrically into $\mathbb H^3$ with curvature $-K$.  In hyperbolic space,
  this sphere bounds a ball of volume $O(n)$.  Retriangulating that ball gives
  us our bounded geometry filling.

  Replacing the star of each vertex in $\mathcal V_2$ with a bounded geometry
  filling of its link gives us our bounded geometry triangulation of $M$.
\end{proof}
Notice that this strategy generalizes to higher dimensions: if every bounded
geometry triangulated $S^k$ of volume $V$ can be filled with a bounded geometry
disk of volume $F_k(V)$, then every $n$-manifold $M$ with a triangulation of
unbounded geometry and volume $V$ has a triangulation of bounded geometry and
volume
\[G_n(V)=F_{n-1}(F_{n-2}(\cdots F_3(O(V))\cdots)).\]
However, we currently have no upper bounds on the functions $F_k$ for $k \geq 3$.  Finding such estimates seems to be a worthwhile and difficult problem.  For now, we show the first nonlinear \emph{lower} bound for this type of problem:
\begin{thm}
  The function $F_{8n-1}$ is at least quadratic for all $n \geq 1$.
\end{thm}
In particular, in light of Theorem \ref{prelim-main}, the geometrically simplest nullcobordisms of bounded geometry $(8n-1)$-spheres are not always disks.
\begin{proof}
  We show this by constructing a sequence of PL triangulations $T_N$ of
  $S^{8n-1}$ with $O(N)$ simplices and locally bounded geometry, such that disks
  of locally bounded geometry filling them must have $\Omega(N^2)$ simplices.
  We construct these as boundaries of disks in manifolds homotopy equivalent to
  $S^{4n} \times S^{4n}$.

  Let $p:(X,S^{4n-1} \times D^{4n}) \to (D^{4n}, S^{4n-1})$ be the unit disk
  bundle of a vector bundle with Pontryagin number $p_n=1$ relative to the
  boundary.  Let $Y_0$ be a thickening of
  $S^{4n} \vee S^{4n} \subset S^{4n} \times S^{4n}$ and produce a manifold $Y_N$
  with boundary by puncturing each thickened copy of $S^{4n}$ $N$ times and
  gluing in copies of $X$.  Then $Y_N$ has a bounded-degree PL triangulation
  with $O(N)$ simplices: one can use $O(1)$ simplices for each copy of $X$ and
  $O(1)$ simplices for a neighborhood of the wedge point, and $O(N)$ simplices
  to connect these together.  Moreover, since $p$ has zero Euler class, the
  induced sphere bundle is trivial, and hence $\partial Y_N$ is topologically a
  sphere.  In general it is an exotic sphere; however:
  \begin{lem}
    There is an integer $p>0$ such that $\partial Y_N$ is the standard
    sphere when $p \operatorname{|} N$.
  \end{lem}
  \begin{proof}
    Coning off $\partial Y_N$ gives a topological manifold $M_N$ homotopy
    equivalent to $S^{4n} \times S^{4n}$.  From the surgery exact sequence it
    immediately follows that such structures are classified by their $p_1$
    class lying in $\ZZ \oplus \ZZ$.  By \cite{OnSS}, the set of smoothable such
    structures contains a lattice.  When $(N,N)$ is in this lattice,
    $\partial Y_N$ must be the standard sphere.
  \end{proof}

  When $\partial Y_N$ is the standard sphere, gluing in a disk gives a smooth
  manifold $M_N$.  By construction, its $n$th Pontryagin class evaluates to $N$
  on each $S^{4n}$ factor, and hence the Pontryagin number $p_{n,n}=2N^2$.
  Then by the Hirzebruch signature formula, since all other Pontryagin numbers
  are zero, we must have $p_{2n} \sim N^2$.  By \cite{LR}, Pontryagin classes
  can be computed via local combinatorial formulas, and therefore any
  bounded-geometry triangulation of $M_N$ must have $\Omega(N^2)$ simplices,
  with constants depending on the choice of bound on geometry.  That means that
  for any fixed such choice, any bounded-geometry disk filling our
  triangulation of $\partial Y_N$ will have $\Omega(N^2)$ simplices.
\end{proof}

\subsection{General pseudomanifolds}
The next most obvious application of our theorem after PL manifolds is to
pseudomanifolds with arbitrary singularities.  Notice that the corresponding
problem with unbounded geometry is trivial, as one can fill any pseudomanifold  by taking its cone.  On the other hand, restating our main theorem in
this case gives the following new result:
\begin{cor}
  Let $Y$ be a finite complex.  Then every homologically trivial simplicial
  $k$-cycle in $Y$ of bounded geometry of type $L$ and volume $V$ has a filling
  of bounded geometry of type $L'=L'(k,L)$ and volume $f(V)$, where
  $f(V)=O(V^{1+\epsi})$ for every $\epsi>0$.
\end{cor}
\begin{proof}
  Given a homologically trivial cycle, match each facet of one of its simplices
  to an equal facet with opposite orientation.  Identifying the matched facets
  creates a simplicial set.  Taking a single barycentric subdivision turns this
  simplicial set into a (not necessarily connected) pseudomanifold to which the
  main theorem can be applied.
\end{proof}

This fact may be useful in the study of high-dimensional expanders, random
complexes, and related topics.

\subsection{Other classes}
There are a number of other categories of manifolds with singularities whose
bordism theory is relevant in geometric topology.  These include Witt spaces
\cite{Siegel,Goresky}, whose bordism homology theory is closely related to
KO-theory.  Many other examples are discussed in \cite[\S5.2.1]{Friedman}.

Some such examples, including Witt spaces, have infinitely generated bordism groups.  In this case, the corresponding classifying spaces do not have finite type.  However, for a given bound on geometry, bordism classes are classified by a subspace with finite skeleta.  In other words, one can find bordism classes which require arbitrarily high local complexity, as expressed by the topology of links of simplices.  Contrast this with the case of PL manifolds, in which all links have the same topology.  Our main theorem still holds as written in such a setting, although for any given bound $L$ on the geometry, we are really only considering part of the bordism theory.

This suggests that, unlike in the case of PL manifolds, the problem of finding null-bordisms for Witt spaces without bounded geometry is fundamentally different.

On the other hand, note that if, for example, one is interested in Witt
null-bordisms of non-singular PL manifolds, this can be done with finitely many
topological link types, even without a bounded geometry assumption on the
original manifold.

\subsection{Relation to secondary invariants}

One area of application of these ideas is, in principle, to the connection between ``secondary'' invariants and complexity.  The most famous secondary invariants---the $\eta$-invariants of Atiyah--Patodi--Singer \cite{APS,APS2,Wall}---and their $L^2$ analogues due to Cheeger and Gromov are linearly bounded as a function of the volume of a manifold with bounded geometry \cite{CheeGr}.  These invariants are defined using the signature of a nullcobordism, and indeed, this inequality was Gromov's original motivation for suggesting that cobordisms may have linear volume.  Since our volume estimates are superlinear, they do not recover the optimal asymptotics of Cheeger and Gromov.

Cha \cite{Cha}, in dimension 3, and Lim and Weinberger \cite{LW}, in all dimensions, have proven an inequality for the Cheeger--Gromov $\rho$-invariant in terms of the number of simplices by a mixture of algebraic and geometric methods.  The followup paper of Cha and Lim \cite{CL} makes essential use of a non--locally finite classifying space for proving such estimates, and strongly suggests the possibility of sometimes extending the results of this paper beyond bounded geometry. 

For other invariants, our results do provide a new estimate.  For example, the bordism homology group $\Omega_k^{\mathcal M}(B\Gamma)$, when $\mathcal M$ is the class of Witt spaces, is used to define higher $\rho$-invariants of manifolds with fundamental group $\Gamma$ \cite{rho}.  If $B\Gamma$ is of finite type, our
results can then be used to bound these higher $\rho$-invariants, and therefore (in certain situations) to bound the number of manifolds of bounded geometry homotopy equivalent to a given one.  On the other hand, if $B\Gamma$ is not of finite type, we get an important example of a situation in which our results do not apply.

\begin{rmk}
  The $\eta$-invariants can also be studied in terms of a very different and much cruder complexity measure: number of handles, rather than number of simplices, in a nullcobordism.  For example, in dimension three, it is the classical Heegaard genus.  Unlike the case of simplices, there can be infinitely many manifolds with boundedly many handles, e.g.~there are infinitely many 3-manifolds with Heegaard genus 1.  Gromov suggests \cite[\S5$\frac{5}{7}$.I]{GrPC} that in many cases the number of handles necessary to produce an oriented nullcobordism of a compact manifold can be bounded linearly in the number of handles of the manifold itself.  We have only been able to verify this assertion for framed manifolds in dimension $\not\equiv 3 \mod 4$.\footnote{This is a consequence of the fact that the main theorems of surgery theory are proved by a greedy method of improvement, special facts about quadratic forms over $\mathbb Z$, and the theorem of Kervaire--Milnor showing that the number of differentiable structures on the $n$-sphere is finite.}  On the other hand, for manifolds with fundamental group $\mathbb Z/2\mathbb Z$, one can use the $\rho$-invariant to give a lower bound on the number of handles in a bordism over $\ZZ/2\ZZ$, and therefore deduce that the number of handles in such a nullcobordism cannot be bounded at all by the number of handles in the manifold.  We investigate this further in a joint article with Bena Tshishiku \cite{morse}.
\end{rmk}

\section{Efficient Whitney embeddings}

In this section we describe a way of embedding a PL $k$-manifold efficiently into
$\mathbb{R}^n$, where $n \geq 2k+1$, implementing step (1) of the proof outline.

In \cite{GrGu}, Gromov and Guth describe ``thick'' embeddings of $k$-dimensional
simplicial complexes in unit $n$-balls, for $n \geq 2k+1$.  They define the
\emph{thickness} $T$ of an embedding to be the maximum value such that disjoint
simplices are mapped to sets at least distance $T$ from each other.  In
\cite[Thm.~2.1]{GrGu}, given a complex with volume $V$ and bounded geometry of type $L$, they construct embeddings with a lower bound on thickness whose asymptotic
behavior as a function of $V$ is close to the theoretical upper bound.

Their construction is probabilistic: they first produce a random embedding, in
which some simplices may pass too close to each other.  They then bend the
simplices around each other on a smaller scale to resolve these near-collisions.

Gromov and Guth's result was strengthened in \cite[Appendix A \S2]{CDMW}, using
essentially the same method, to give the same bound with respect to a slightly
stronger notion of bounded geometry.  We give a version of this strengthened result here.  Denote the link of a simplex $\sigma$ by $\lk\sigma$.
\begin{thm}[{based on \cite[Appendix A, Theorem 2.1]{CDMW}}] \label{thm:emb}
  Suppose that $X$ is a $k$-dimensional simplicical complex with $V$ vertices
  and each vertex lying in at most $L$ simplices.  Suppose that $n \geq 2k+1$.
  Then there are constants $C(n,L)$ and $\alpha(n,L)>0$ and a subdivision $X'$
  of $X$ which embeds linearly into the $n$-dimensional Euclidean ball of radius
  \[R \leq C(n,L)V^{\frac{1}{n-k}}(\log V)^{2k+2}\]
  such that:
  \begin{enumerate}[(i)]
  \item The embedding has Gromov--Guth thickness 1.
  \item For any $i$-simplex $\sigma$ of $X'$, the induced embedding
    $\lk\sigma \to S^{n-i-1}$ into the round unit sphere has Gromov--Guth
    thickness $\alpha(n,L)$.
  \item Adjacent vertices of $X'$ are mapped to points at distance
    $\leq \ell(n,L)$.
  \item The vertices of the embedding are ``snapped to a grid'': they are points
    in a lattice in $\mathbb{R}^n$ which depends only on $L$.
  \end{enumerate}
\end{thm}
\begin{rmks}
  \begin{enumerate}[(a)]
  \item When $X$ is a PL $k$-manifold, an embedding is \emph{locally flat} if it
    is locally PL homeomorphic to the standard embedding of $\mathbb{R}^k$ in
    $\mathbb{R}^n$.  By the Zeeman unknotting theorem, this is always true
    unless $n-k=2$.
  \item Conditions (i) and (iii) taken together imply that each simplex is
    $b$-bilipschitz to a standard simplex for some $b=b(n,L)$.
  \item Together with condition (iv), this means that the simplices in the
    embedding come from a finite number of isometry types depending only on $n$,
    $k$ and $L$.
  \end{enumerate}
\end{rmks}
\begin{proof}
  The cited theorem gives an embedding of $X$ satisfying conditions (i) and
  (ii).  Moreover, the proof given in \cite{CDMW} produces a subdivision whose
  simplices are $b(n,L)$-bilipschitz to an equilateral simplex whose edges have
  length $(\log V)^{2k+2}$.  A further subdivision gives simplices which are
  bilipschitz to the standard simplex at scale $1$.  In particular, this gives
  condition (iii).

  To see that condition (iv) can be satisfied, first take an embedding that
  satisfies conditions (i)--(iii) and rescale it to have Gromov--Guth thickness
  $1+\alpha(n,L)$ (thus increasing the constant $\ell(n,L)$).  Now move all the
  vertices to the nearest point of a cubic lattice of side length
  $\lambda=\alpha(n,L)/4\sqrt{n}$, extending the map linearly to the rest of
  the complex.  Then each vertex, and therefore each point of the complex,
  moves by at most $\alpha(n,L)/8$, so the resulting complex still has
  Gromov--Guth thickness at least $1$.  Moreover, since vertices are at least
  $1$ apart, the angular movement of each point in a link is at most
  $\alpha(n,L)/2$, and therefore the thickness of links in the resulting
  complex is still at least $\alpha(n,L)/2$.  Thus we can accommodate condition
  (iv) by at worst halving $\alpha(n,L)$.
\end{proof}

\section{Simplicial Pontrjagin--Thom constructions} \label{S:thom}

Since cobordism theories are cohomology theories, we can represent a cobordism
class by a map to a classifying space and a nullbordism by a nullhomotopy of
such a map.  It is somewhat difficult to describe such classifying spaces in an
explicit way, although \cite[\S II.5]{BRS} and \cite{LR} provide possible
approaches.  In this section, starting with a category $\mathcal M$ of PL
manifolds with prescribed singularities (in other words, of pseudomanifolds with
a prescribed set of permissible links), we develop a ``Pontrjagin--Thom
construction'': a space $\Th\mathcal M(n,k)$ such that for $n$ sufficiently
larger than $k$, $\mathcal M$-bordism classes of spaces in $\mathcal M$
correspond to homotopy classes of maps $S^n \to \Th\mathcal M(n,k)$.  Note that
we do not guarantee that this mapping is bijective, but merely injective:
$\pi_n(\Th\mathcal M(n,k))$ may contain classes that are not in the image of the
bordism group $\Omega^{\mathcal M}_k$.  Therefore the spaces we construct are not
explicit classifying spaces for $\mathcal M$-bordism; we fail to answer,
e.g.,~\cite[\S8,~Question~2]{Goresky}.  Nevertheless, this construction is
sufficient for our purpose.

\subsection{Classes of singularities} \label{S:singularities}
We first specify the necessary conditions on the category $\mathcal M$ that are required for bordism to make sense.  These axioms are similar to those given in \cite[\S IV.3]{BRS}.  A category of manifolds-with-singularity is specified by a family $\{\mathcal L_n\}_{n \in \mathbb N}$ of sets of $(n-1)$-dimensional simplicial complexes which constitute permissible links of vertices in a closed $n$-dimensional $\mathcal M$-manifold.  These sets must satisfy:
\begin{enumerate}
\item $\mathcal L_n$ is closed under PL isomorphism.
\item \label{A:lower} Each member of $\mathcal L_n$ is a closed $\mathcal M$-manifold (that is, its links are contained in $\mathcal L_{n-1}$).
\item \label{A:susp} If $\LL \in \mathcal L_n$, then the suspension
  $S\LL \in \mathcal L_{n+1}$.
\item If $\LL \in \mathcal L_n$, then the cone $c\LL \notin \mathcal L_{n+1}$.
\end{enumerate}
The last axiom means that there is a notion of $\mathcal M$-manifold with boundary and that the boundary of such an object is well-defined.  Namely, an $n$-dimensional $\mathcal M$-manifold with boundary is a simplicial complex whose links lie in $\mathcal L_n$ or $c\mathcal L_{n-1}$, and its boundary is the subcomplex of points whose links lie in $c\mathcal L_{n-1}$.  Axiom \ref{A:susp} ensures that if $M$ is an $\mathcal M$-manifold, then $M \times I$ is a $\mathcal M$-manifold with boundary.  Axiom \ref{A:lower} implies, among other things, that the link of a $k$-simplex of an $\mathcal M$-manifold lies in $\mathcal L_{n-k}$.

An additional ``regularity'' axiom, which is not strictly necessary, guarantees that $\mathcal M$-manifolds are pseudomanifolds:
\begin{enumerate}[resume]
\item $\mathcal L_1=\{S^0\}$.
\end{enumerate}

\subsection{Thom spaces}
The construction of the Thom space is straightforward: it is patched together out of pieces which encode possible local behaviors of a $k$-dimensional $\mathcal M$-manifold embedded simplexwise linearly in $\mathbb R^n$.  To be precise, $\Th\mathcal M(n,k)$ is a semi-simplicial complex glued out of subcomplexes $K_{\LL,f,N}$ (which we call \emph{vertex gadgets}) corresponding to triples $(\LL,f,N)$, where:
\begin{itemize}
\item $\LL$ is a permissible link in $\mathcal M$.
\item $f:c\LL \to \mathbb R^n$ is a linear embedding of the cone on $\LL$.  Let $p_0$ be the cone point of $c\LL$.
\item $N$ is a neighborhood of $f(p_0)$ which is equipped with a linear triangulation transverse to $f$ and such that $f^{-1}(N) \cap \LL=\emptyset$.
\item For every $k$-simplex $\Delta \subset \LL$, the subcomplex $N_{c\Delta} \subseteq N$ defined by
  \[N_{c\Delta}=\bigcup \{\text{simplices }\sigma \subset N \mid \emptyset \neq f^{-1}(\sigma) \subseteq \Int\st(c\Delta \subset cP)\}\]
  satisfies $N_{c\Delta} \cap f(c\Delta) \neq \emptyset$.  In other words, there is at least one simplex of $N$ which intersects $f(c\Delta)$ nontrivially and does not intersect $f(c\Delta')$ for any simplex $\Delta' \subset \LL$ which does not contain $\Delta$.  
\end{itemize}
\begin{figure}
  \centering
  \begin{subfigure}[b]{0.45\textwidth}
    \centering
  \begin{tikzpicture}[z={(0.0,1.0)},x={(0.9,0.4)},y={(-0.75,0.85)}]
    \coordinate (A) at (0,-1,0);
    \coordinate (B) at (2,-0.5,0.2);
    \coordinate (C) at (1.5,0.5,0);
    \coordinate (D) at (1,1,1);
    \coordinate (E) at (-1.6,1,1);
    \coordinate (F) at (-1,0.5,0);
    \coordinate (G) at (-2.2,-0.2,0.2);
    \coordinate (O) at (0,0,0);
    \filldraw[fill=gray!20!white, very thick] (A)--(B)--(C)--(D)--(E)--(F)--(G)--cycle;
    \draw[very thick] (A)--(O)--(B) (C)--(O)--(D) (E)--(O)--(F) (O)--(G);
    \coordinate (X1) at (-.913,-0.3,0.083);
    \coordinate (X2) at (-.85,.1,.05);
    \coordinate (X3) at (-.8,.46,.3);
    \coordinate (X4) at (-.7,.6,.6);
    \coordinate (X5) at (0.1,0.7,0.7);
    \coordinate (X6) at (0.3,0.702,0.702);
    \coordinate (X6pt5) at (0.9,-0.45,.09);
    \coordinate (X7) at (0.8,-0.5,0.08);
    \coordinate (X8) at (0.42,-.53,.042);
    \coordinate (X8pt5) at (-.22,-.52,.02);
    \coordinate (X9) at (-.44,-.5,.04);
    \coordinate (Y1) at (-.7,-.3,.6);
    \coordinate (Y2) at (-.5,.3,.8);
    \coordinate (Y3) at (.8,.4,.9);
    \coordinate (Y4) at (1.1,-.02,0.7);
    \coordinate (Y5) at (0.7,-0.2,.7);
    \coordinate (Y6) at (0,-0.2,.7);
    \coordinate (Y7) at (0,-0.4,0.3);
    \fill[fill=gray!50!white, opacity=0.8] (X1) -- (-.88,-.08,.08) -- (X2) -- (-.963,.4815,0) -- (X3) -- (-.723,.452,.452) -- (X4) -- (X5) -- (X6) -- (Y3) -- (Y4) -- (X6pt5) -- (X7) -- (X8) -- (0,-.536,0) -- (X8pt5) -- (X9) -- cycle;
    \draw (X1) -- (Y1) -- (X9)
      (X8) -- (Y7) -- (Y1) -- (Y2) -- (X2)
      (X3) -- (Y2) -- (X4)
      (X5) -- (Y2) -- (Y3) -- (Y6) -- (Y2)
      (Y1) -- (Y6) -- (Y5) -- (X7)
      (Y3) -- (Y5) -- (Y4) -- (X6pt5)
      (Y5) -- (Y7) -- (Y6) (Y7) -- (X8pt5)
      (Y4) -- (Y3) -- (X6);
  \end{tikzpicture}
  \caption{Here, $f(cP)$ is light gray with thick lines and
    $N$ is dark gray with thin lines.}
\end{subfigure}
\quad
\begin{subfigure}[b]{0.45\textwidth}
  \centering
  \begin{tikzpicture}z={(0.0,1.0)},x={(0.9,0.4)},y={(-0.75,0.85)}]
    \coordinate (A) at (0,-1,0);
    \coordinate (B) at (2,-0.5,0.2);
    \coordinate (C) at (1.5,0.5,0);
    \coordinate (D) at (1,1,1);
    \coordinate (E) at (-1.6,1,1);
    \coordinate (F) at (-1,0.5,0);
    \coordinate (G) at (-2.2,-0.2,0.2);
    \coordinate (O) at (0,0,0);
    \filldraw[fill=gray!20!white, very thick] (A)--(B)--(C)--(D)--(E)--(F)--(G)--cycle;
    \draw[very thick] (A)--(O)--(B) (C)--(O)--(D) (E)--(O)--(F) (O)--(G);
    \coordinate (X1) at (-.913,-0.3,0.083);
    \coordinate (X2) at (-.85,.1,.05);
    \coordinate (X3) at (-.8,.46,.3);
    \coordinate (X4) at (-.7,.6,.6);
    \coordinate (X5) at (0.1,0.7,0.7);
    \coordinate (X6) at (0.3,0.702,0.702);
    \coordinate (X6pt5) at (0.9,-0.45,.09);
    \coordinate (X7) at (0.8,-0.5,0.08);
    \coordinate (X8) at (0.42,-.53,.042);
    \coordinate (X8pt5) at (-.22,-.52,.02);
    \coordinate (X9) at (-.44,-.5,.04);
    \coordinate (Y1) at (-.7,-.3,.6);
    \coordinate (Y2) at (-.5,.3,.8);
    \coordinate (Y3) at (.8,.4,.9);
    \coordinate (Y4) at (1.1,-.02,0.7);
    \coordinate (Y5) at (0.7,-0.2,.7);
    \coordinate (Y6) at (0,-0.2,.7);
    \coordinate (Y7) at (0,-0.4,0.3);
    \fill[fill=gray!50!white, opacity=0.8] (X1) -- (-.88,-.08,.08) -- (X2) -- (-.963,.4815,0) -- (X3) -- (-.723,.452,.452) -- (X4) -- (X5) -- (X6) -- (Y3) -- (Y4) -- (Y5) -- (Y6) -- (Y1) -- cycle;
    \fill[fill=gray!80!white, opacity=0.8] (X1) -- (Y1) -- (Y6) -- (Y5) -- (Y4) -- (X6pt5) -- (X7) -- (X8) -- (0,-.536,0) -- (X8pt5) -- (X9) -- cycle;
    \draw (X1) -- (Y1) -- (X9)
      (X8) -- (Y7) -- (Y1) -- (Y2) -- (X2)
      (X3) -- (Y2) -- (X4)
      (X5) -- (Y2) -- (Y3) -- (Y6) -- (Y2)
      (Y1) -- (Y6) -- (Y5) -- (X7)
      (Y3) -- (Y5) -- (Y4) -- (X6pt5)
      (Y5) -- (Y7) -- (Y6) (Y7) -- (X8pt5)
      (Y4) -- (Y3) -- (X6);
    \node[inner sep=1.3pt,circle,draw=black,fill=black,thick] at (A) {};
    \node[right] at (A) {$v$};
  \end{tikzpicture}
  \caption{Now we highlight a vertex $v$ and the corresponding subcomplex
    $N_{cv}$.}
\end{subfigure}
\caption{Schematic illustrations of $f(cP)$, $N$ for a typical triple $(P,f,N)$.}
\end{figure}
We identify two triples $(\LL,f,N)$ and $(\LL,f',N')$ if $f$ and $N$ differ from $f'$ and $N'$ by the same translation.

Given such a triple, let $N_\infty \subset N$ be the subcomplex consisting of simplices disjoint from $f(c\LL)$.  Then the vertex gadget $K_{\LL,f,N}$ is isomorphic to the semi-simplicial complex $N/N_\infty$, and $\Th\mathcal M(n,k)$ is glued out of all the vertex gadgets using the following gluing relations:
\begin{itemize}
\item All the points resulting from collapsing various $N_\infty$ are identified (this is the ``point at infinity'' in the Thom space).
\item Given $K_{\LL,f,N}$ and $K_{\LL',f',N'}$, we identify the subcomplexes
  \[N_{c\Delta} \subset K_{\LL,f,N} \quad\text{and}\quad N'_{c\Delta'} \subset K_{\LL',f',N'}\]
  if (1) $N_{c\Delta}=N'_{c\Delta'}$ as complexes embedded in $\mathbb R^n$ and
  (2) there is an isomorphism $\iota:\st(c\Delta) \to \st(c\Delta')$ such that
  $f' \circ \iota = f$.
\end{itemize}

Optionally, we can make this into a Thom space for oriented $\mathcal M$-manifolds by adding orientation data for top-dimensional simplices of $c\LL$ to the triple $(\LL,f,N)$, and require that this data match in order to identify subcomplexes.

Notice that $\Th\mathcal M(n,k)$ occurs as a subcomplex in $\Th\mathcal M(n+1,k+1)$.  This is because for any vertex gadget $K_{\LL,f,N}$ of $\Th\mathcal M(n,k)$, $\Th\mathcal M(n+1,k+1)$ contains a vertex gadget $K_{\LL',f',N'}$ corresponding to the triple
\[(\LL'=\partial(c\LL \times [0,1]), f':(x,t) \mapsto (f(x),3t-1), N'=N \times [0,1]),\]
with $\LL'$ and $N'$ triangulated via a subdivision of the product cell structure so that:
\begin{enumerate}
\item The inclusions of $c\LL \times 0$ and $N \times 0$ are simplicial.
\item All instances of $\Delta^k \times [0,1]$ are subdivided in the same
  (permutation-independent) way.
\end{enumerate}
Then $K_{\LL,f,N}$ includes into $K_{\LL',f',N'}$ as the subcomplex $K_{cv}$ for
\[v=(\text{cone point},0) \in \partial(c\LL \times [0,1]).\]
These inclusions respect the gluings that produce the complexes as a whole.  If using oriented simplices, we require that $c\LL$ face in the positive direction in $c\LL \times [0,1]$.

Finally, we must show that these complexes are indeed Thom spaces in the sense we need:
\begin{prop} \label{prop:thom}
  When $n>2k+1$, there is an injective ``Pontrjagin--Thom'' homomorphism
  \[\Omega^{\mathcal M}_k \to \pi_n(\Th\mathcal M(n+1,k+1))\]
  where $\Omega^{\mathcal M}_k$ is the bordism group of $k$-dimensional
  $\mathcal M$-manifolds.  Moreover, $\Th\mathcal M(n,k))$ is
  $(n-k-1)$-connected.
\end{prop}
\begin{proof}
  Given an $\mathcal M$-manifold $X$, we can construct a PL embedding
  $f:X \to \mathbb R^n$.  We then construct a suitable triangulation of
  $\mathbb R^n$ by letting $A=f(X)$ in the following lemma.
  \begin{lem} \label{lem:fine}
    Let $A$ be a finite simplicial complex linearly embedded in $\mathbb R^n$.
    There is a linear triangulation $\tau$ of $\mathbb R^n$ which is transverse
    to $A$ and has the following property.  For a vertex $v \in A$, define the
    subset
    \[N_v=\bigcup \{\Delta \in \tau : \emptyset \neq \Delta \cap A \subseteq \Int\st(v)\}.\]
    Then the sets $N_v$ cover $A$.
  \end{lem}
  \begin{proof}
    Let $\bs(A)$ be the barycentric subdivision of $A$.  It suffices for $\tau$
    to be such that if a simplex $\Delta$ of $\tau$ intersects the star in
    $\bs(A)$ of a vertex $v$, then $\tau \cap A$ is contained in the open star
    of $v$ in $A$.  This is true for any sufficiently fine triangulation.
  \end{proof}
  Now suppose we have such a triangulation.  Then for any simplex $\Delta$ of
  $X$,
  \[N_\Delta=\bigcap_{v \in \Delta} N_v\]
  satisfies the following conditions:
  \begin{itemize}
  \item $f^{-1}(N_\Delta)$ is contained in $\Int(\st(\Delta))$, and in fact for
    each $v \in \Delta$, $N_\Delta$ is the maximal subcomplex of $N_v$ for which
    this holds.
  \item $N_\Delta \cap f(\Delta) \neq \emptyset$ (by the
    Knaster--Kuratowski--Mazurkiewicz lemma, since for any simplex
    $\Sigma \subset \Delta$, $\bigcup_{v \in \Sigma} N_v$ covers $f(\Sigma)$).
  \end{itemize}
  Therefore, there is a well-defined map $\mathbb R^n \to \Th\mathcal M(n,k)$
  which sends each $N_v$ to the subcomplex $K_{\lk(v),f|_{\st(v)},N_v}$, and
  simplices disjoint from $f(X)$ to the point at infinity.  This map sends a
  simplex $\sigma'$ of $\tau$ which intersects $f(X)$ nontrivially to the
  corresponding simplex $\sigma'$ of $K_{\lk(v),f|_{\st(v)},N_v}$, and every other
  simplex to the point at infinity.  This shows that every $k$-dimensional
  $\mathcal M$-manifold induces a corresponding element of
  $\pi_n(\Th \mathcal M(n+1,k+1))$.

  Similarly, if $X$ is the boundary of a $(k+1)$-dimensional
  $\mathcal M$-manifold $W$, we can extend the embedding $f$ to an embedding
  $W \to \mathbb R^n \times [0,1]$ and use this embedding to construct a
  nullhomotopy $D^{n+1} \to \Th\mathcal M(n+1,k+1)$ of $f$.  This shows that the
  correspondence is well-defined on bordism classes. 

  Conversely, given a simplicial map $g:S^n \to \Th\mathcal M(n+1,k+1)$, we can
  recover a (possibly empty) $k$-dimensional $\mathcal M$-manifold.  Notice that
  $\Th \mathcal M(n,k)$ contains a ``zero section''
  \[\Th_0\mathcal M(n,k)=\bigcup_{\LL,f,N} (f(cP) \cap N \subset K_{\LL,f,N}).\]
  \begin{lem}
    $X=g^{-1}(\Th_0\mathcal M(n+1,k+1))$ is an $\mathcal M$-manifold.
  \end{lem}
  \begin{proof}
    Let $x$ be a point in $X \cap \Int(g^{-1}(\sigma))$, where $\sigma$ is an
    $(n-j)$-simplex of $\Th\mathcal M(n+1,k+1)$.  In order to show that $X$ is
    an $\mathcal M$-manifold, we need to understand the link of $x$ in $X$.
    The simplex $\sigma$ is contained in some set of vertex gadgets
    $K_{\LL,f,N}$, and the link of $g(x)$ in $\sigma \cap \Th_0\mathcal M(n,k)$
    is
    \[\bigcup_{K_{\LL,f,N} \supset \sigma} (\text{link of $g(x)$ in }f(c\LL) \cap N).\]
    For each choice of $(\LL,f,N)$, $g(x)$ is contained in the interior of an
    $\ell$-simplex $\Delta$ (where $\ell \geq j+1$) of $f(c\LL)$ whose link $L$
    in $f(c\LL)$ is, by the definition of the gluing relation, independent of
    the choice of $(\LL,f,N)$.  Moreover, by Axiom \ref{A:lower} for admissible
    links in $\mathcal M$, $L$ is an admissible link for
    $(k+1-\ell)$-dimensional $\mathcal M$-manifolds.

    We claim that the link of $x$ in $X$ is the $(\ell-1)$-fold suspension of
    $L$.  Since this is true for every $x \in X$, $X$ is an
    $\mathcal M$-manifold.

    To see the claim, notice that since $f(c\LL)$ is transverse to $N$, for any
    simplex $\sigma'$ containing $x$, there is a neighborhood of $x$ in
    $X \cap \sigma'$ which is foliated by copies of $cL$.  This is true
    regardless of the dimensions of the simplices $\sigma'$ and $g(\sigma')$.
  \end{proof}
  Similarly, for a simplicial map $D^{n+1} \to \Th\mathcal M(n+1,k+1)$, the
  preimage of the zero section is a $(k+1)$-dimensional $\mathcal M$-manifold
  with boundary.

  We have now constructed well-defined maps
  \[\Omega^{\mathcal M}_k \to \pi_n(\Th\mathcal M(n+1,k+1)) \to \Omega^{\mathcal M}_k\]
  whose composition is the identity.  Therefore the first of these maps is
  injective.

  Now notice that for every subcomplex $K_{(\LL,f,N)}$ of $\Th\mathcal M(n,k)$,
  the $(n-k-1)$-skeleton of $N$ does not intersect the image of $f$, that is,
  it is contained in $N_\infty$.  Since all the $N_\infty$ are collapsed to the
  point at infinity in $\Th\mathcal M(n,k)$, its $(n-k-1)$-skeleton is a point,
  and so it is $(n-k-1)$-connected.
\end{proof}


\subsection{Modifications for bordism homology} \label{S:addY}
Here we describe the modifications to the construction of $\Th\mathcal M(n,k)$
required to encode bordism classes of maps from $\mathcal M$-manifolds to a
simplicial complex $Y$.  We define a space $(\Th \wedge Y)\mathcal M(n,k)$ which is built out of pieces encoding possible local behaviors of a $k$-dimensional $\mathcal M$-manifold embedded simplexwise linearly in $\mathbb R^n$ \emph{and additionally equipped with a map to $Y$}.  We construct the space out of the same vertex gadgets as before, but give them an additional index: the vertex gadget $K_{(\LL,f,N,g)}$, where $g:c\LL \to Y$ is a simplicial map, is isomorphic to $N/N_\infty$.  As before, the cone points are all identified as one point at infinity, and we identify $K_{\Delta} \subset K_{(\LL,f,N,g)}$ and $K_{\Delta'} \subset K_{(\LL',f',N',g')}$ if
\begin{enumerate}
\item $N_{c\Delta}=N'_{c\Delta'}$ as complexes embedded in $\mathbb R^n$ and
\item there is an isomorphism $\iota:\st(\Delta) \to \st(\Delta')$ such that
  $f' \circ \iota = f$ \emph{and} $g' \circ \iota=g$.
\end{enumerate}

Note that $(\Th \wedge Y)\mathcal M(n,k)$ comes with a natural projection map
\[p:(\Th \wedge Y)\mathcal M(n,k) \to \Th\mathcal M(n,k)\]
which is finite-to-one if $Y$ is a finite complex.  Let
\[(\Th \wedge Y)_0\mathcal M(n,k)=p^{-1}\Th_0\mathcal M(n,k).\]
\begin{prop} \label{prop:smash}
  When $n>2k+1$, there is an injective map
  \[\Omega^{\mathcal M}_k(Y) \to \pi_n((\Th \wedge Y)\mathcal M(n+1,k+1))\]
  where $\Omega^{\mathcal M}_k(Y)$ is the $k$th $\mathcal M$-bordism group
  of $Y$.  Furthermore, $(\Th \wedge Y)\mathcal M(n,k)$ is
  $(n-k-1)$-connected.
\end{prop}
\begin{proof}
  Given a triangulated $\mathcal M$-manifold $X$ and a simplicial map
  $\ph:X \to Y$, we construct a map $S^n \to \Th\mathcal M(n,k)$ as in
  Proposition \ref{prop:thom}, taking care that the embedding
  $X \hookrightarrow \mathbb R^n$ is linear on a subdivision $\tau$ of our
  triangulation.  To lift this map to $(\Th \wedge Y)\mathcal M(n,k)$, we need
  consistent choices of $g:c\lk(v) \to Y$ for each vertex $v$ of $\tau$.  We
  obtain this by homotoping $\ph$ to a map which is simplicial on $\tau$.  We
  use a similar method to convert a filling of $(X,\ph)$ to a nullhomotopy.
  This gives a well-defined map
  \[\Omega^{\mathcal M}_k(Y) \to \pi_n((\Th \wedge Y)\mathcal M(n+1,k+1)).\]
  It remains to show that this map is injective.  We will do this once again by
  constructing a retraction.

  Given a simplicial map $g:S^n \to (\Th \wedge Y)\mathcal M(n+1,k+1)$, we can
  construct a manifold
  \[X=g^{-1}((\Th \wedge Y)_0\mathcal M(n+1,k+1))\]
  as in Proposition \ref{prop:thom} after forgetting the data about $Y$.  Now we
  need to build a map from this manifold to $Y$ which, when $g$ is the Thom map
  of some $\ph:X \to Y$, agrees up to homotopy with $\ph$.

  We do this by composing the induced map
  $i:X \to (\Th \wedge Y)\mathcal M(n+1,k+1)$ with a map $\tilde g$ from the
  zero section $(\Th \wedge Y)_0\mathcal M(n+1,k+1)$ to $Y$.  This will be a
  simplicial map from a triangulation of the zero section to the one-time
  barycentric subdivision $\bs(Y)$, and it is constructed by induction on
  dimension.  For a simplex $\Delta$ of $Y$, denote by $\Delta^\vee$ the dual
  subcomplex to $\Delta$ in $\bs(Y)$; that is, if $v_\Delta$ is the barycenter
  of $\Delta$, then $\st(v_\Delta)=\Delta^\vee * \partial \Delta$.  For each
  $K_{(P,f,N,g)}$ we build $\tilde g:f^{-1}(N) \to Y$ as follows:
  \begin{itemize}
  \item For each $(k+1)$-simplex $\Delta$ of $c\LL$, set
    $\tilde g(N_\Delta)=v_{g(\Delta)}$.
  \item For each $k$-simplex $\Delta$ of $c\LL$, extend $\tilde g$ to $N_\Delta$
    so that its composition with the projection to $\Delta^\vee$ is homotopic to
    $g$ via a homotopy which restricts to the linear contraction on
    $N_{\Delta'}$ for each of the $(k+1)$-simplices $\Delta'$ incident to
    $\Delta$.  So that $\tilde g$ is well-defined, this map should depend only
    on $N_\Delta$ and $g|_{\st \Delta}$.
  \item Continue this process for lower-dimensional simplices
    $\Delta \subset \LL$, at all points fixing maps in such a way that they
    depend only on $N_\Delta$ and $g|_{\st \Delta}$.
  \end{itemize}
  If we start with $(X,\ph)$, build the corresponding map
  $S^n \to (\Th \wedge Y)\mathcal M(n,k)$, and then use this to construct
  $\tilde g \circ i$ as above, then by construction
  $\tilde g \circ i \simeq \ph$.  Therefore, as in Proposition \ref{prop:thom},
  we have a retraction
  \[\Omega^{\mathcal M}_k(Y) \to \pi_n((\Th \wedge Y)\mathcal M(n+1,k+1)) \to \Omega^{\mathcal M}_k(Y),\]
  and therefore the first map is injective.

  Finally $(\Th \wedge Y)\mathcal M(n,k)$ is $(n-k-1)$-connected by the same
  reasoning as in Proposition \ref{prop:thom}.
\end{proof}

\section{The Thom map and its nullhomotopy}

In this section, we give the details of steps (2) and (3) of our outline: first use the embedding constructed in Theorem \ref{thm:emb} to build a geometrically controlled Thom map to a finite subcomplex of $\Th\mathcal M(n,k)$, and then find a controlled nullhomotopy of this map in a larger finite subcomplex of $\Th \mathcal M(n+1,k+1)$. 
\begin{lem} \label{lem:phi}
  For every $n$, $k$, $L$, and $\mathcal M$, there is a finite, $(n-k-1)$-connected subcomplex $T_{\mathcal M}(n,k,L) \subset \Th\mathcal M(n,k)$ such that the following holds.  Let $X$ be a PL triangulated $k$-dimensional $\mathcal M$-manifold with $V$ $k$-simplices and bounded geometry of type $L$, and let $n \geq 2k+1$.  Then there is a PL map $f:S^n \to T_{\mathcal M}(n,k,L)$ such that $f^{-1}(\Th_0\mathcal M(n,k))$ is PL homeomorphic to $X$ and the Lipschitz constant of $f$ (with respect to the standard simplexwise linear metric on $T_{\mathcal M}(n,k,L)$) is at most
  \[C(n,L)V^{\frac{1}{n-k}}(\log V)^{2k+2}.\]
  Moreover, if $Y$ is a finite simplicial complex and $\ph:X \to Y$ is a
  simplicial map, then $f$ lifts to a PL map
  \[\tilde f:S^n \to p^{-1}T_{\mathcal M}(n,k,L) \subset (\Th \wedge Y)\mathcal M(n,k)\]
  such that $\tilde g \circ \tilde f|_{\tilde f^{-1}\Th_0\mathcal M(n,k)} \simeq \ph$,
  where $\tilde g:(\Th \wedge Y)_0\mathcal M(n,k) \to Y$ is the projection
  defined in the proof of Proposition \ref{prop:smash} and
  $p:(\Th \wedge Y)\mathcal M(n,k) \to \Th\mathcal M(n,k)$ is the projection
  defined at the beginning of \S\ref{S:addY}.
\end{lem}
\begin{proof}
  Let $\Lambda(L)$ be the lattice in $\mathbb R^n$ specified in condition (iv) of Theorem \ref{thm:emb}.  Let $\tau$ be a $\Lambda(L)$-invariant triangulation of $\mathbb R^n$ with the following properties:
  \begin{itemize}
  \item The edge lengths are at most $\frac{\alpha(n,L)}{2}$.
  \item It is transverse to every possible simplex of an embedding satisfying conditions (i)--(iv) of Theorem \ref{thm:emb}.  (This is possible since there are finitely many such possible simplices up to the action of $\Lambda(L)$.)
  \end{itemize}

  Now let $T_{\mathcal M}(n,k,L) \subseteq \Th\mathcal M(n,k)$ be the union of
  subcomplexes $K_{(\LL,f,N)}$ such that:
  \begin{itemize}
  \item $\LL$ is an admissible link in $\mathcal L_n$.
  \item $f$ is an embedding of $c\LL$ satisfying conditions (i)--(iv) of Theorem
    \ref{thm:emb}.
  \item $N$ is the subcomplex of $\tau$ consisting of simplices which intersect
    $f(c\LL)$ and are at distance at least $\frac{\alpha(n,L)}{2}$ from $f(\LL)$.
  \end{itemize}
  The number of such combinations is finite, and therefore $T_{\mathcal M}(n,k,L)$
  is a finite complex.  Moreover, it's $(n-k-1)$-connected since all the
  $K_{(\LL,f,N)}$ are $(n-k-1)$-connected and glued together along
  $(n-k-1)$-connected subcomplexes.

  Now by Theorem \ref{thm:emb}, there is a subdivision $X'$ of $X$ which embeds
  linearly into the $n$-dimensional Euclidean ball of radius
  \[R=C(n,L)V^{\frac{1}{n-k}}(\log V)^{2k+2}\]
  such that the embedding satisfies conditions (i)--(iv).  Conditions (i) and
  (ii) imply that the triangulation $\tau$ satisfies the conditions of Lemma
  \ref{lem:fine}.  This embedding induces a map from the $R$-ball to
  $T_{\mathcal M}(n,k,L)$ which is simplicial on a $C(n)$-times barycentric
  subdivision of $\tau$, and hence $C(n,L)$-Lipschitz.  Since the boundary of
  the $R$-ball is mapped to $\infty$, this map extends to a sphere.  If we
  rescale this to be the unit sphere, the Lipschitz constant becomes $C(n,L)R$.

  Finally, in the case of a map $\ph:X \to Y$, we can homotope this map to a map
  $\ph':X \to Y$ which is simplicial on $X'$, and then decide the lift to
  $(\Th \wedge Y)\mathcal M(n,k)$ based on the behavior of $\ph'$ near each
  vertex.
\end{proof}

Note that for any finite simplicial complex $Y$, since $p$ is finite-to-one,
$p^{-1}T_{\mathcal M}(n,k,L)$ is also a finite complex.  We now embed
$p^{-1}T_{\mathcal M}(n,k,L)$ in a larger, but still finite subcomplex of
$(\Th \wedge Y)\mathcal M(n+1,k+1)$ in which we can nullhomotope what needs to be
nullhomotoped.  This construction is essentially purely formal.
\begin{lem} \label{lem:U}
  There is a finite simplicial complex $U_{\mathcal M,Y}(n,k,L) \supseteq p^{-1}T_{\mathcal M}(n,k,L)$ equipped with a simplicial map
  \[\iota:U_{\mathcal M}(n,k,L) \to (\Th \wedge Y)\mathcal M(n+1,k+1)\]
  which extends the inclusion $\iota_0:p^{-1}T_{\mathcal M}(n,k,L) \hookrightarrow (\Th \wedge Y)\mathcal M(n,k)$, such that:
  \begin{enumerate}[(i)]
  \item $U_{\mathcal M,Y}(n,k,L)$ is $(n-k-1)$-connected.
  \item If a map $S^n \to p^{-1}T_{\mathcal M}(n,k,L)$ is nullhomotopic in
    $(\Th \wedge Y)\mathcal M(n+1,k+1)$, then it is nullhomotopic in
    $U_{\mathcal M,Y}(n,k,L)$.
  \end{enumerate}
\end{lem}
\begin{proof}
  We know that $(\Th \wedge Y) \mathcal M(n+1,k+1)$ and
  $p^{-1}T_{\mathcal M}(n,k,L)$ are both $(n-k-1)$-connected.  We will build
  $U_{\mathcal M,Y}(n,k,L)$ by adding cells of dimension $n+1$ and therefore not
  change this.

  Since $p^{-1}T_{\mathcal M}(n,k,L)$ is a simply connected finite complex, its
  higher homotopy groups are finitely generated.  In particular, the kernel of
  the induced map
  \[(\iota_0)_*:\pi_n(p^{-1}T_{\mathcal M}(n,k,L)) \to \pi_n((\Th \wedge Y)\mathcal M(n+1,k+1))\]
  is finitely generated.  Choose simplicial maps $f_i:S^n \to p^{-1}T_{\mathcal M}(n,k,L)$ representing the generators of the kernel, as well as simplicial fillings
  \[F_i:D^{n+1} \to (\Th \wedge Y)\mathcal M(n+1,k+1).\]
  (In this notation, we are suppressing the simplicial structures on the sphere and disk on which these maps are defined.)

  We build the simplicial complex $U_{\mathcal M,Y}(n,k,L)$ by gluing these
  simplicial disks onto $p^{-1}T_{\mathcal M}(n,k,L)$ using the maps $f_i$ as
  attaching maps.  Then the map $\iota$ is constructed by extending $\iota_0$ to
  the new cells using the maps $F_i$.
\end{proof}

Now put the standard simplexwise linear metric on $U_{\mathcal M,Y}(n,k,L)$; this
also restricts to a metric on $p^{-1}T_{\mathcal M}(n,k,L)$.  Then
\cite[Corollary 4.3]{CDMW} yields the following fact:
\begin{lem} \label{lem:psi}
  Suppose that $n \geq 2k+2$, and let $f:S^n \to p^{-1}T_{\mathcal M}(n,k,L)$ be an
  $L$-Lipschitz map such that $\iota_0 \circ f$ is nullhomotopic.  Then $f$ is
  nullhomotopic in $U_{\mathcal M,Y}(n,k,L)$ via a $C(L+1)$-Lipschitz nullhomotopy,
  where $C$ is a constant which depends on $\mathcal M$, $n$, $k$, $L$, and $Y$.
\end{lem}

\section{Completing the proof}
We now put the earlier results together to prove the main theorem, which we
restate here.  This section implements steps (4) and (5) of the proof outline.
\begin{thm*}
  Let $\mathcal M$ be a category of manifolds with prescribed singularities, $k$
  and $L$ integers, $\epsi>0$, and $Y$ a finite simplicial complex.  Then there
  are constants $C=C(\mathcal M,k,L,\epsi,Y)$ and $L'=L'(\mathcal M,k,L)$ such
  that the following holds.

  Let $X$ be a triangulated $k$-dimensional $\mathcal M$-manifold with $V$
  $k$-simplices and bounded geometry of type $L$, and let $\ph:X \to Y$ be a
  simplicial map such that $(X,\ph)$ represents the zero class in
  $\Omega^{\mathcal M}_k(Y)$.  Then there is a $(k+1)$-dimensional
  $\mathcal M$-manifold $W$ with $\partial W=X$ with at most $CV^{1+\epsi}$
  simplices and bounded geometry of type $L'$, and a simplicial map
  $\tilde \ph:W \to Y$ extending $\ph$.
\end{thm*}
\begin{proof}
  Let $n \geq 2k+2$.  The Thom construction from Proposition \ref{prop:smash}
  tells us that the simplicial map $f:S^n \to p^{-1}T_{\mathcal M}(n,k,L)$
  induced by any PL embedding of $X$ is nullhomotopic in
  $(\Th \wedge Y)\mathcal M(n+1,k+1)$.  Combining Lemmas \ref{lem:phi} and
  \ref{lem:psi}, we see that there is a map $F:D^{n+1} \to U_{\mathcal M,Y}(n,k,L)$
  such that $F|_{S^n}=f$ and $\lambda=\Lip F=O(V^{1/q}(\log V)^{2k+2}))$.

  By a quantitative simplicial approximation theorem, see
  e.g.~\cite[Prop.~2.1]{CDMW}, we can simplicially approximate $F$ on a
  subdivided triangulation of $D^{n+1}$ at scale $\sim 1/\lambda$.  This gives us
  a simplicial nullhomotopy of a simplicial approximation $\tilde f$ of $f$.

  We would like to get back to a nullhomotopy of $f$.  Notice that
  $\tilde f=f \circ \sigma$ where $\sigma:S^n \to S^n$ is a simplicial map
  from the subdivision to the original triangulation.  Moreover, $\tilde f$ can
  be constructed so that for each simplex $\delta$ of the original
  triangulation, $\sigma^{-1}(\delta)$ is a contractible subcomplex.  So we fix
  a triangulation of $S^n \times [0,1]$ as follows.  First triangulate
  $S^n \times [0,1]$ using the usual product triangulation, where
  top-dimensional simplices take the form
  \[[(\delta_0,0), (\delta_1,0), \ldots, (\delta_k,0), (\delta_k,1), (\delta_{k+1},1), \ldots, (\delta_n,1)]\]
  for each $0 \leq k \leq n$ and each $n$-simplex $[\delta_0,\ldots,\delta_n]$
  of the subdivided triangulation of $S^n$.  Now take the mapping cylinder of
  $\sigma$, which is a simplicial quotient of this triangulation.  Our extra
  assumption about $\tilde f$ implies that the mapping cylinder is still a
  triangulation of $S^n \times [0,1]$, with the subset $S^n \times \{0\}$ given
  the subdivided triangulation and $S^n \times \{1\}$ given the original
  triangulation.  Then the map $\sigma \circ \text{proj}_{S^n}$ on the product
  triangulation induces a simplicial homotopy between $f$ and $\tilde f$ on the
  mapping cylinder.  So we get a simplicial nullhomotopy of $f$ by attaching
  this to our disk as a collar.

  Now Lemma \ref{lem:U} gives us a simplicial map to our model of the
  classifying space for the bordism homology of $Y$,
  \[\iota:U_{\mathcal M,Y}(n,k,L) \to (\Th \wedge Y)\mathcal M(n+1,k+1).\]
  Then by the argument of Proposition \ref{prop:smash},
  $(\iota \circ F)^{-1}(\Th \wedge Y)_0\mathcal M(n,k)$ is an
  $\mathcal M$-manifold $W$ whose boundary is $X$.  Since the image of $\iota$
  is a finite subcomplex of $(\Th \wedge Y)\mathcal M(n+1,k+1)$, the link types
  of $W$ come from a finite set depending on $\mathcal M$, $n$, $k$, $L$, and
  $Y$.  In particular, there is some $L'(\mathcal M,n,k,L,Y)$ which bounds the
  geometry of these link types.

  Moreover, there is a triangulation of $W$ with a bounded number of simplices
  per simplex of $D^{n+1}$: we obtain it by pulling back from the image under
  $\iota \circ F$ the triangulation on which the map
  \[\tilde g:(\Th \wedge Y)_0\mathcal M(n,k) \to Y\]
  is defined; incidentally this gives us a simplicial map
  $\Phi=\tilde g \circ \iota \circ F:W \to Y$.  With this triangulation, $W$
  consists of $O\bigl(V^{1+\frac{k+1}{n-k}}(\log V)^{(2k+2)(n+1)}\bigr)$ simplices.
  Given any $\epsi>0$, for sufficiently large $n$, this is $O(V^{1+\epsi})$.  

  At this point in the proof, $L'$ depends on $n$ and $Y$ as well as $k$, $L$,
  and the category.  We will fix this at the expense of increasing the constant
  $C$, by an inductive process which replaces certain subcomplexes of $W$ with
  other subcomplexes.  The strategy is similar to the proof of Corollary
  \ref{cor:M3}.

  We start by taking a single barycentric subdivision of $W$.  Let $V_i$ be the
  set of vertices in the interior (outside $\partial W=X$) of the resulting
  triangulation that originate as barycenters of $i$-simplices.  Notice that the
  star of a vertex $v \in V_i$ is a join
  $\bs(\partial\Delta^i) * c\bs(\LL(v))$, where $\LL(v)$ is the
  $(k-i-1)$-dimensional link of the original $i$-simplex.  We will modify $W$ by
  replacing $c\bs(\LL(v))$ with other fillings of $\bs(\LL(v))$.

  We start with $i=k-2$.  If $\mathcal M$-manifolds are pseudomanifolds, then
  for $v \in V_{k-2}$, $\bs(\LL(v))$ is a disjoint union of circles.  This can be
  filled with disks with $7$ faces to a vertex, at most $3$ faces to a vertex on
  the boundary, and linearly many faces in terms of $\lvert \LL(v) \rvert$; let
  $Q(v)$ be this collection of disks.  For each $v \in V_{k-2}$, we replace its
  star $\bs(\partial\Delta^{k-2}) * c\bs(\LL(v)) \subset W$ with
  $\bs(\partial\Delta^{k-2}) * Q(v)$ to get a new $\mathcal M$-manifold $W_2$.
  Notice that although we have eliminated the vertices in $V_{k-2}$, the sets of
  vertices $V_i$ for $i<k-2$ are unchanged from $W$, and their stars are of the
  form $\bs(\partial\Delta^i) * c\LL_2(v)$, where $\LL_2(v)$ is constructed as
  follows.  For $v \in V_{k-3}$, $\bs(\LL(v))$ is tiled by $c\bs(\LL(w))$'s for
  $w \in V_{k-2}$, meeting three to a vertex; $\LL_2(v)$ is built by replacing
  each $c\bs(\LL(w))$ by $Q(w)$.  Then the parts of $\LL_2(v)$ outside $X$ have
  at most $9$ faces meeting at a vertex.

  More generally, suppose by induction that we have built a complex $W_{i-1}$ in
  which we have replaced the stars of vertices $w \in V_{k-(i-1)}$ with
  subcomplexes of the form $\bs(\partial\Delta^{i-1}) * Q(w)$, where $Q(w)$ is a
  subcomplex of bounded geometry of type some $L'_{i-1}(\mathcal M,k,L)$.  Then
  for each $v \in V_{k-i}$, its star is a complex
  $\bs(\partial\Delta^{k-i}) * c\LL_{i-1}(v)$ such that the geometry of the
  $(i-1)$-complex $\LL_{i-1}(v)$ is bounded by
  $L_i=\max(L,iL'_{i-1}(\mathcal M,k,L))$.  Moreover, by induction and since we
  had an original bound on the complexity of link types, the volume of
  $\LL_{i-1}(v)$ is bounded as a function of $\mathcal M$, $n$, $k$, $L$, and
  $Y$.  We can then replace $c\LL_{i-1}(v)$ with a complex $Q(v)$ (a fixed
  $\mathcal M$-manifold whose boundary is $\LL_{i-1}(v)$) such that:
  \begin{itemize}
  \item The geometry of $Q(v)$ is bounded by $L'_i=L'(\mathcal M,2i,i-1,L_i,*)$.
  \item The volume of $Q(v)$ is bounded as a function of $\mathcal M$, $n$, $k$,
    $L$, and $Y$ (simply because the set of possible $\LL_{i-1}(v)$s is finite
    and depends only on these parameters).
  \end{itemize}
  This completes the inductive step to construct $Y$.

  At each stage, we extend the map $\Phi$ over $Q(v)$ by mapping each interior
  vertex to $\Phi(v)$ and extending by linearity.  Since we are replacing the
  star of $v$, this is a well-defined map.

  At the end of the induction, we have replaced $W$ with an
  $\mathcal M$-manifold whose boundary is $X$ and whose geometry still depends
  on $\mathcal M$, $k$, and $L$, but not on $n$.  Moreover, we have multiplied
  the volume by at most a constant depending on $\mathcal M$, $n$, $k$, and $L$.
\end{proof}

\appendix
\section{PL versus smooth bounded geometry} \label{smoothing}

Here we discuss the relationship between bounded geometry for smooth and for PL
manifolds.  Specifically, we show that for smooth manifolds these are closely
related notions.  We have been informed by Januszkiewicz that decades ago these results were lectured on by Calabi at an {\it Arbeitstagung}.

A \emph{smoothing} of a PL manifold $X$ is a smooth structure which is compatible with the PL structure.  This can be formalized, for example, as an atlas of smooth charts whose transition maps to the original PL atlas are piecewise smooth.  Through dimension 6, the smooth and PL categories are equivalent; in higher dimensions, the smoothing problem has a homotopy-theoretic reformulation: a smoothing is determined by a homotopy class of lifts of the classifying map $X \to BPL_n$ to $BO_n$ \cite[Part II]{HM}.  The obstructions to performing this lift are the groups $\Theta_k$ of exotic $k$-spheres, for $k \leq n$.  We say a PL manifold is \emph{smoothable} if it admits a smoothing.

\begin{thm} \label{thm:triangulation}
  Let $M$ be a Riemannian $n$-manifold with sectional curvatures $|K| \leq 1$
  and injectivity radius at least $1$, and volume $V$.  Then there is a
  triangulation of $M$ with at most $C_1(n)V$ simplices and at most $C_2(n)$
  simplices incident to each vertex.
\end{thm}
\begin{thm} \label{thm:smoothing}
  Let $M$ be a smoothable PL $n$-manifold equipped with a PL triangulation with
  at most $L$ simplices incident to any vertex.  Then for constants
  $C_1,C_2,C_3$ depending on $n$ and $L$, every smoothing of $M$ has an
  associated Riemannian metric $g$ such that:
  \begin{enumerate}
  \item Every simplex of $(M,g)$ has volume at most $C_1$.
  \item $(M,g)$ has sectional curvature $|K| \leq C_2$ and injectivity radius at
    least $C_3^{-1}$.
  \end{enumerate}
\end{thm}
The tools needed to prove Theorem \ref{thm:triangulation} are given in
\cite{CMS}; here we give an outline, referring to that paper for the more
difficult parts.
\begin{proof}[Proof of Theorem \ref{thm:triangulation}.]
  We start by fixing an $\epsi$-net of points $\{x_1,\ldots,x_N\}$ on $M$, for
  some $\epsi<1/2$ possibly depending on $n$: by definition, the $\epsi$-balls
  around the $x_i$ are disjoint and the $2\epsi$-balls cover $M$.  The bounded
  geometry condition implies that the exponential map onto each of these balls
  is a diffeomorphism and not too distorted.  Moreover, the lower bound on
  curvature implies that $B_{4\epsi}(x_i)$ intersects at most some $C_0(n)$ of
  the $4\epsi$-balls around the other $x_j$.  Thus there is a
  $(C_0+1)$-coloring of the $x_i$ into sets $A_0,\ldots,A_{C_0}$ so that no two
  $4\epsi$-balls of the same color intersect.

  We can triangulate each $B_{3\epsi}(x_i)$ with the image under the exponential
  map of some standard Euclidean mesh; the distortion of each simplex in each of
  these triangulations is bounded by some constant depending only on $n$.  Now
  \cite[Lemma 6.3]{CMS} shows that we can take any two such local boundedly
  distorted meshes and connect them by a mesh which still has bounded (if
  somewhat worse) distortion.  We apply this lemma first to interpolate between
  the triangulations of $B_{3\epsi}(x_i)$ and $B_{3\epsi}(x_j)$ for $x_i \in A_0$
  and $x_j \in A_1$; then to add in the $3\epsi$-balls around the points in
  $A_2$; and so on.  We thus worsen the geometry of the triangulation a total of
  $C_0$ times.  In the end, we get a triangulation of all of $M$ for which the
  geometry of the simplices still depends only on $n$.
\end{proof}
Theorem \ref{thm:smoothing} follows easily from the fact that the groups
$\Theta_n$ of exotic spheres are all finite (except perhaps in dimension $4$,
where any exotic spheres are not standard PL spheres).
\begin{proof}[Proof of Theorem \ref{thm:smoothing}.]
  Let $(M,\tau)$ be a triangulated $n$-manifold, and suppose that $\tau$ has
  at most $L$ simplices adjacent to every vertex.  In particular, the star of
  each $k$-simplex can have a finite number $a_k$ of possible configurations.

  We will show that there is a finite set $\Gamma$ of Riemannian metrics on the
  $n$-simplex, depending only on $n$ and $L$, such that every smoothing of $M$
  has a metric which, when restricted to each simplex, is isometric to one of
  the metrics in $\Gamma$.  This is sufficient to prove the theorem because the
  constants $C_1$, $C_2$, and $C_3$ are then obtained by maximizing over the
  finite set of possible local configurations.

  We construct $\Gamma$ by induction, giving a process that builds all possible
  smoothings of $M$ as Riemannian manifolds with metrics on simplices chosen
  from $\Gamma$.  First let $\tau'$ be the two times barycentric subdivision of
  $\tau$, and let
  \[M_0 \subset M_1 \subset \cdots \subset M_n=M\]
  be open submanifolds of $M$ such that each $M_k$ includes the star of the
  $k$-skeleton of $\tau$ in $\tau'$.  Then $M_0$ is a disjoint union of balls
  around the vertices of $\tau$, and going from $M_{k-1}$ to $M_k$ means gluing
  in copies of $D^k \times D^{n-k}$ for each $k$-simplex of $M$.

  Clearly $M_0$ has a unique smoothing.  For every possible link of a vertex
  consisting of at most $L$ simplices, we fix a Riemannian metric on its cone;
  this gives a Riemannian metric on $M_0$.

  Now suppose we have a smoothing of $M_{k-1}$ equipped with a Riemannian metric
  $g$, in which each $n$-simplex of $\tau'$ is isometric to one of a finite list
  depending on $n$ and $L$.  For every $k$-simplex $\Delta$ of $\tau$, if the
  smoothing extends to $\st_{\tau'}(\Delta)$, then the set of such extensions is
  in bijection with the finite group $\Theta_k$ of exotic $k$-spheres.  We fix
  Riemannian metrics on all these extensions, depending on the combinatorial
  structure of $\st_{\tau'}(\Delta)$, the metric on
  $\st_{\tau'}(\partial \Delta)$, and the element of $\Theta_k$.  By induction,
  this data takes a finite set of values depending only on $k$ and $L$.

  After the $n$th step, for every smoothing of $M$, we have obtained a metric
  with these properties.
\end{proof}

\bibliographystyle{amsalpha}
\bibliography{pl}
\end{document}